\documentclass[11pt]{amsart}
\usepackage[a4paper,margin=1.2in]{geometry}
\usepackage{amsmath,amssymb,amsthm}
\usepackage{xcolor}
\usepackage[colorlinks=true,linkcolor=blue,citecolor=red,urlcolor=blue]{hyperref}

\newtheorem{theorem}{Theorem}

\newtheorem{obs}[theorem]{Observation}
\newcommand{\titex}[1]{\textit{#1}}
\newtheorem{lemma}[theorem]{Lemma}

\newtheorem{conjecture}[theorem]{Conjecture}

\usepackage{xcolor}
\usepackage[bottom]{footmisc}

\begin{document}

\title{Restricting Odd Antiholes in Critical Graphs via List Coloring}
 \author{Alireza Fiujlaali}
    \author{Jessica McDonald}
    \address{Department of Mathematics and Statistics, Georgia State University, Atlanta, GA USA 30303.}
    \email{afiujlaali1@gsu.edu}
    \address{Department of Mathematics and Statistics,  Auburn University, Auburn, AL 36849. Supported in part by Simons Foundation Grant \#845698 and NSF grant DMS-2452103.}
    \email{mcdonald@auburn.edu}
\date{}

\begin{abstract}
By the Strong Perfect Graph Theorem, every graph whose chromatic number exceeds its clique number contains an odd hole or an odd antihole. We investigate the structure of odd antiholes in critical graphs with fixed maximum degree. We prove that if $G$ is a $k$-critical graph with maximum degree $\Delta$ then every odd antihole in $G$ has length at most $2(\Delta-k)+7$, and conjecture that, in the odd-hole-free setting, equality forces the graph itself to be the corresponding odd antihole.  We verify this extremal characterization when $\Delta\geq 2k-5$, extending recent results for $\Delta=k$, and construct $k$-critical odd-hole-free graphs showing that the next possible odd antihole length $2(\Delta-k)+5$ below the extremal bound can occur. Our proofs use a list-coloring theorem for graphs on $2\ell-1$ vertices whose complements are hypomatchable. This motivates the conjecture that every odd antihole $\overline{C_{2\ell-1}}$ is strongly $\ell$-chromatic-choosable; we show that this conjecture would imply the full extremal characterization. We also obtain degree and attachment restrictions for odd antiholes in critical graphs.
\end{abstract}

\maketitle

\section{Introduction}

In this paper all graphs are finite and simple. For terms not defined here we refer the reader to \cite{West}.

A \titex{hole} in a graph is an induced cycle of length at least 4, and an \emph{antihole}
is an induced copy of $\overline{C_t}$ (the complement of $C_t$) for some $t\geq 4$. We say that $t$ is the length of the antihole $\overline{C_t}$; holes and antiholes are \emph{odd} if their length is odd. Since $\overline{C_5}=C_5$, the graph $C_5$ is both an odd hole and an odd antihole, but otherwise the concepts are distinct, and they are closely tied to chromatic graph theory via the \emph{Strong Perfect Graph Theorem}. All graphs $G$ trivially satisfy $\chi(G)\geq \omega(G)$, where $\chi(G), \omega(G)$ are the \emph{chromatic number} and \emph{clique number} of $G$, respectively; the Strong Perfect Graph Theorem of Chudnovsky, Robertson, Seymour and Thomas \cite{SPGT} implies that when $\chi(G)>\omega(G)$, $G$ must contain an odd hole or odd antihole. In this paper we investigate odd antiholes in \emph{critical} graphs, that is, graphs $G$ where $\chi(G-v)<\chi(G)$ for every $v\in V(G)$, and find that they are quite restricted. As a first step towards this consider the following observation. Here, by \emph{$k$-critical} we mean a critical graph with chromatic number equal to $k$.

\begin{obs}\label{obs:2k-1} Let $G$ be a $k$-critical graph. Then the length of any odd antihole is at most $2k-1$, and $G=\overline{C_{2k-1}}$ is the unique graph achieving this bound.
\end{obs}

This observation follows immediately from the facts that $\chi(\overline{C_{2s-1}})=s$ and $\overline{C_{2s-1}}$ is $s$-critical for all $s\geq 3$. In this paper we seek an analog of Observation \ref{obs:2k-1} for graphs with fixed maximum degree; we prove and conjecture as follows.

\begin{theorem}\label{thm:length} Let $G$ be a $k$-critical graph with maximum degree $\Delta$. Then the length of any odd antihole in $G$ is at most $2(\Delta-k)+7$.
\end{theorem}

\begin{conjecture}\label{conj:main}  Let $G$ be an odd-hole-free  $k$-critical graph with maximum degree $\Delta$. If $G$ contains an odd antihole of length $2(\Delta-k)+7$, then  $G=\overline{C_{2(\Delta-k)+7}}$.
\end{conjecture}

\begin{theorem}\label{thm:2k-5} Conjecture \ref{conj:main} is true when  $\Delta\leq k$ and when $\Delta \geq 2k-5$.
\end{theorem}
 The meaningful cases of Theorem \ref{thm:2k-5} are the equality cases. Indeed, if $G$ contains an extremal odd antihole $A=\overline{C_{2(\Delta-k)+7}}$ and $G\neq A$, then criticality gives $k\geq\chi(A)+1=\Delta-k+5$, and hence $\Delta\leq2k-5$, making $2k-5$ a natural threshold.
The $\Delta<k$ case of Theorem \ref{thm:2k-5} is trivial by Brooks' Theorem, since the only graphs under consideration there are cliques and odd cycles, the former of which has no odd antihole, and the latter of which is itself an odd hole. The $\Delta=k$ case of Theorem \ref{thm:2k-5} is dealt with by the following theorem of Chen, Lan, Lin, and Zhou \cite{CLLZ} and Galindo, McDonald and Shan \cite{GMS}: \emph{If $G$ is an odd-hole-free $\Delta$-critical graph with maximum degree $\Delta$ then $G= \overline{C_7}$.}  Our above theorems and conjecture may be viewed as extending this extremal phenomenon from $k=\Delta$ to a general $k$ and $\Delta$. Note that $\overline{C_7}$ is the shortest odd antihole that is not also an odd hole, which is why the $k=\Delta$ case is particularly tidy. Note also that  Conjecture \ref{conj:main} is fully verified for $k\leq 6$ since in this case the open  range of the conjecture, namely $k<\Delta< 2k-5$, is empty.

One might wonder if the next value below the length bound of $2(\Delta-k)+7$ in Theorem \ref{thm:length} can actually be achieved, or if there is an improvement that could be gained when the extremal case is excluded. The following theorem answers this in the negative. In fact, it also answers in the negative the analogous question for Observation \ref{obs:2k-1}.

\begin{theorem}\label{thm:sharp} For every $k\geq 5$, there exists a $k$-critical odd-hole-free graph $G$ with maximum degree $\Delta=2k-4$, and containing an induced odd antihole $\overline{C_{2k-3}}=\overline{C_{2(\Delta-k)+5}}$, but no shorter induced odd antihole.
\end{theorem}

Beyond length, another way that odd antiholes in critical graphs are restricted is that most of their vertices must have large degree. Every $k$-critical graph has minimum degree at least $k-1$ (else we can delete a low-degree vertex $v$ and extend the $(k-1)$-coloring of $G-v$ greedily to $v$). We prove the following.

\begin{theorem}\label{thm:degrees} Let $G$ be a $k$-critical graph with maximum degree $\Delta$, and suppose $G$ contains an odd antihole $\overline{C_{2t+1}}$ with $t\geq 3$ (i.e., not $\overline{C_5}$). Then the number of vertices in the antihole with degree at least $k$ in $G$ is at least $t+1$ if $t\geq 4$ and at least $t$ if $t=3$. Moreover, for $1\leq j\leq t-3$,  the number of vertices in the antihole with degree at least $k+j$ in $G$ is at least  $\min\{t-j-1, 2t+k-\Delta-2\}$.
\end{theorem}

The threshold $k+t-3$ in Theorem \ref{thm:degrees} cannot be increased in general: $G=\overline{C_{2t+1}}$ is a $(t+1)$-critical graph where every vertex has degree exactly $2t-2=k+t-3$. 

We obtain one further structural restriction on odd antiholes, concerning how vertices outside the antihole may attach to it without creating an odd hole. We will discuss this technical result more before we prove Theorem \ref{thm:sharp}, as it will be useful there.

Our proofs of Theorems \ref{thm:length} and \ref{thm:degrees} use the notion of \emph{list-coloring}. A \titex{list assignment} for a graph $G$ assigns a set $L(v)$ of colors to every $v\in V(G)$; a graph $G$ is \titex{$L$-colorable} if it has a coloring $\varphi$ such that $\varphi(v)\in L(v)$ for every $v\in V(G)$. A graph $G$ is \emph{$\ell$-list-colorable} if it is $L$-colorable for every list-assignment $L$ with $|L(v)|\geq \ell$ for all $v\in V(G)$. The following theorem from 2015 verified a conjecture of Ohba~\cite{Ohba}.

\begin{theorem}[Noel, Reed, and Wu \cite{NRW}]\label{thm:Ohba} Every $\ell$-chromatic graph $G$ with $|V(G)|\leq 2\ell+1$ is $\ell$-list-colorable.
\end{theorem}

If $G$ is $\ell$-chromatic then we certainly cannot hope for better than $\ell$-list-colorability, given the case of identical lists of size $\ell-1$. However, we show that, in a very special situation, other lists of size $\ell-1$ can work. Here, by \emph{hypomatchable} we mean a graph $H$ where $H-v$ has a perfect matching for every $v\in V(H)$. By $\ell$-assignment we mean a list-assignment where all lists have size exactly $\ell$.

\begin{theorem}\label{thm:list} Let $\ell\geq 3$ be an integer and let $G$ be a graph with $|V(G)|=2\ell-1$ whose complement is hypomatchable. Let $L$ be an $(\ell-1)$-assignment for $G$ where all lists are subsets of a common $\ell$-set. Then $G$ is $L$-colorable unless all lists in $L$ are identical.
\end{theorem}

If the $G$ in Theorem \ref{thm:list} is additionally $\ell$-chromatic, then $G$ is certainly not $L$-colorable from identical lists of size $\ell-1$. So Theorem \ref{thm:list} shows that, among all list assignments of this form, identical lists are the only obstruction to $L$-colorability.

Our purpose in proving Theorem \ref{thm:list} is to apply it to the $\ell$-chromatic odd antihole $\overline{C_{2\ell-1}}$ (whose complement, a cycle, is hypomatchable), within the proof of Theorem \ref{thm:2k-5}. We conjecture that the ``common palette'' assumption of Theorem \ref{thm:list} should be removable. This is particularly appealing because we can show it would imply the remaining range of  \ref{conj:main}. A graph $G$ is said to be \emph{strongly $\ell$-chromatic-choosable} if $\chi(G)=\ell$ and every $(\ell-1)$-assignment for which $G$ is not list-colorable has the property that the lists are the same for all vertices; see Kaul and Mudrock~\cite{KM} for more on this concept.

\begin{conjecture}\label{conj:skcc} For every integer $\ell\geq 3$, the odd antihole $\overline{C_{2\ell-1}}$ is strongly $\ell$-chromatic-choosable.
\end{conjecture}

\begin{theorem}\label{thm:conjs} Conjecture \ref{conj:skcc} implies Conjecture \ref{conj:main}. In particular, if Conjecture \ref{conj:skcc} holds for $\ell$, then Conjecture \ref{conj:main} holds for every pair $\Delta, k$ with $\Delta-k=\ell-4$.  
\end{theorem}

The case $\ell=3$ of Conjecture \ref{conj:skcc} is the well-known fact that a cycle with lists of size 2 is colorable unless all lists are equal: choose consecutive vertices $u, v$ with $L(u)\neq L(v)$, color $u$ with a color not in $L(v)$, and color the remaining vertices greedily along the cycle, ending at $v$. This $\ell=3$ case of Conjecture \ref{conj:skcc} corresponds to the $k=\Delta+1$ case of Conjecture \ref{conj:main}, which is trivial; the $\ell=4$ case of Conjecture \ref{conj:skcc} corresponds to $k=\Delta$ for Conjecture \ref{conj:main}, which is handled by the aforementioned result of Chen et al. and Galindo et al. The first case of Conjecture \ref{conj:skcc} with new consequences is $\ell=5$: by Theorem \ref{thm:conjs}, verifying that $\overline{C_9}$ is strongly 5-chromatic-choosable would establish Conjecture \ref{conj:main} whenever $k=\Delta-1$, and in particular for $(k, \Delta)=(7,8)$, the first open case of Conjecture \ref{conj:main}.

This paper proceeds as follows. Section 2 concerns list-coloring and contains a proof of Theorem \ref{thm:list} in addition to an easier list-coloring result that we will need. We prove Theorems \ref{thm:length} and \ref{thm:degrees} 
in Section 3, using these list-coloring results. We also prove Theorem \ref{thm:conjs} in Section 3. Section 4 contains a proof of Theorem \ref{thm:sharp} as well as the neighborhood result alluded to above, neither of which require list-coloring.

\vspace*{.1in}

\section{List-coloring}

We use the following classic theorem of Hall for our list-coloring proofs. Note that in what follows, $\alpha(G)$ indicates the maximum size of an independent set in a graph $G$.

\begin{theorem}[Hall~\cite{Hall}]\label{thm:hall}
Let $A_1,\ldots,A_n$ be finite sets. There exist distinct elements $a_1,\ldots,a_n$ such that $a_i\in A_i$ for every $i$ if and only if
$
\left|\bigcup_{i\in X}A_i\right|\ge|X|
$
for every $X\subseteq\{1,\ldots,n\}$.
\end{theorem}

\begin{lemma}\label{lem:listmatch}
Let $H$ be a graph, let $A=\overline{H}$, and let $\emptyset\ne S\subsetneq V(A)$ such that $H-S$ is bipartite. Let $L$ be a list assignment for $A$ with $|L(v)|\geq |S|$ for every $v\in S$ and $|L(v)|\geq |S|+\alpha(H-S)$ for all $v\in V(A)\setminus S$. Then $A$ has an $L$-coloring.
\end{lemma}

\begin{proof} Let $s=|S|$ and $q=\alpha(H-S)$. Color the vertices of $S$ greedily with $s$ distinct colors from their lists. Delete all the colors used on $S$ from every list for $V(A)\setminus S$; note these lists still each have size at least $q$ and we aim to color $A-S$ from them.

Since $H-S$ is bipartite, K\H{o}nig's theorem~\cite{Konig} gives $\nu(H-S)=|V(H-S)|-\alpha(H-S)$. Hence, if we take a maximum matching of $H-S$, then the number of matching edges plus the number of unsaturated vertices is exactly $\alpha(H-S)=q$. Let $Q$ be the complete $q$-partite graph whose parts correspond to these pairs of matched vertices (which are non-adjacent in $A$) and these unsaturated vertices. Then $A-S\subseteq Q$.

We prove that $Q$ can be colored from the remaining lists $L'$ by induction on $q$. The result is clear for $q=1$, so assume $q\ge2$. If a two-vertex part $\{x,y\}$ satisfies $L'(x)\cap L'(y)\ne\emptyset$, choose $c\in L'(x)\cap L'(y)$, color $x$ and $y$ with $c$, delete this part, and delete $c$ from every remaining list. The resulting graph has $q-1$ parts and every remaining list has size at least $q-1$, so our result follows by induction. We may therefore assume that the two lists in every two-vertex part are disjoint. For $X\subseteq V(Q)$, let $U(X)=\bigcup_{x\in X}L'(x)$. We show that $|U(X)|\geq |X|$, from which the result follows by Theorem \ref{thm:hall}. If $0<|X|\le q$, then $|U(X)|\ge q\ge|X|$; so we may assume that $|X|>q$. Each color in $U(X)$ appears on at most one vertex from each part and hence on at most $q$ vertices of $X$. Therefore
$
q|X|\le\sum_{x\in X}|L'(x)|\le q|U(X)|,
$
and $|U(X)|\ge|X|$, as desired.
\end{proof}

Note that the final paragraph of the above proof can be replaced by an application of Theorem \ref{thm:Ohba}:  $Q$ is a $q$-chromatic graph with at most $2q$ vertices, so Theorem \ref{thm:Ohba} implies that it is $q$-list-colorable. The inclusion of the above simple argument shows that Lemma \ref{lem:listmatch} is not difficult. We will use a generalized version of this argument to prove the following list-coloring result, which we have restated for convenience.

\setcounter{theorem}{7}
\begin{theorem} Let $\ell\geq 3$ be an integer and let $G$ be a graph with $|V(G)|=2\ell-1$ whose complement is hypomatchable. Let $L$ be an $(\ell-1)$-assignment for $G$ where all lists are subsets of a common $\ell$-set. Then $G$ is $L$-colorable unless all lists in $L$ are identical.
\end{theorem}
\setcounter{theorem}{12}

\begin{proof} Let $P$ be an $\ell$-set from which the $(\ell-1)$-subsets of the list assignment are chosen. For each $v\in V(G)$, let $f(v)$ be the unique color in $P\setminus L(v)$, and suppose that $f$ is not constant. 

Since $\overline{G}$ is hypomatchable, for each $u\in V(G)$ we can choose a perfect matching $M_u$ of $\overline{G}-u$. Moreover, we can partition $V(G)$ into the singleton $u$ and the $\ell-1$ edges of $M_u$ (which are non-edges in $G$). Let $Q_u$ be the complete $\ell$-partite graph induced by these parts. Then $G\subseteq Q_u$; we will show that $Q_u$ is $L$-colorable for some choice of $u$. In fact, we will show that for some choice of $u$, $Q_u$ has an $L$-coloring where each pair of vertices in a part of size 2 gets the same color. To this end, let $L_u(u)=L(u)$ and for any pair of vertices $v,w$ in a part of $Q_u$, let $L_u(\{v, w\})=L(v)\cap L(w)$. Note that in any $Q_u$, the singleton part has a list of size $\ell-1$ and the other parts have lists of size at least $\ell-2$.

Suppose for a contradiction that no $Q_u$ admits such a coloring. Then, given some $u\in V(G)$, Theorem \ref{thm:hall} tells us that there exists a collection $X_u$ of parts of $Q_u$ such that $|X_u|> |\cup_{x\in X_u} L_u(x)|$. Since all lists $L_u$ have size at least $\ell-2$, and since $|X_u|\leq \ell$, we know that $|X_u|\in \{\ell-1, \ell\}$. If $|X_u|=\ell-1$ then we say that $u$ is a \emph{type 1} vertex, and if $|X_u|=\ell$ we say that $u$ is a \emph{type 2} vertex.

If $u$ is a type 1 vertex, then the strict inequality given by Theorem \ref{thm:hall} tells us that\\ ${|\cup_{x\in X_u}L_u(x)|=\ell-2}$ and hence $X_u$ contains every part of $Q_u$ except for the singleton $\{u\}$, and moreover all these lists $L_u$ consist of the same set of $\ell-2$ colors. That is, there exist colors $\alpha_u, \beta_u\in P$ such that, for every edge $vw\in M_u$, $\{f(v), f(w)\}=\{\alpha_u, \beta_u\}$. In particular, this means that exactly $\ell-1$ vertices of $G-u$ have $f$-value equal to each of $\alpha_u, \beta_u$.

If $u$ is a type 2 vertex, then $X_u$ consists of every part of $Q_u$, and the strict inequality given by Theorem \ref{thm:hall} tells us that $|\cup_{x\in X_u} L_u(x)|\leq \ell-1$. So there exists a color $\gamma_u\in P$ such that $f(u)=\gamma_u$ and for every edge $vw\in M_u$,  $\gamma_u\in\{f(v), f(w)\}$. In particular, this means that at least $(\ell-1)+1=\ell$ vertices of $G$ have $f$-value equal to $\gamma_u$.

If two vertices $u, v$ in $G$ are both of type 2, then since $\ell>|V(G)|/2$, we must have that $f(u)=\gamma_u=\gamma_v=f(v)$. In particular, this means that not every vertex in $G$ can be of type 2, since we have assumed that $f$ is not constant. So we may assume that some vertex $z\in V(G)$ is of type 1. 

Consider the colors $\alpha_z, \beta_z, f(z)$. Among the $2\ell-2$ vertices of $G-z$, we know that exactly $\ell-1$ have $f$-value $\alpha_z$ and exactly $\ell-1$ have $f$-value $\beta_z$. Since $\ell\geq 3$, this means that at least two vertices in $G$ have $f$-value equal to each of $\alpha_z, \beta_z$. 

If $f(z)\not\in \{\alpha_z, \beta_z\}$, then no single color is missing from  $\ell$ of the lists $L$, so no vertex in $G$ can be of type 2. But then choose a vertex $a$ with $f(a)=\alpha_z$. Since $a$ must be of type 1, we know $f(v)$ equals only two different colors among all $V(G)\setminus \{a\}$, contradicting the three different $f$-values $f(z), \alpha_z, \beta_z$ which we know to be present.  Hence it must be the case that $f(z)\in \{\alpha_z, \beta_z\}$; suppose without loss that $f(z)=\alpha_z$.

We now have that among the $2\ell-1$ vertices of $G$, exactly $\ell$ have $f$-value $\alpha_z$ and exactly $\ell-1$ have $f$-value $\beta_z$. Choose $b$ with $f(b)=\beta_z$. The vertex $b$ cannot be of type 2, since not enough vertices have $f$-value $\beta_z$. On the other hand, the vertex $b$ cannot be of type 1 either, since that would mean exactly $\ell-1$ vertices of $G-b$ have $f$-value $\alpha_z$ and exactly $\ell-1$ vertices of $G-b$ have $f$-value $\beta_z$, which is not the case. Contradiction. 
\end{proof}

\vspace*{.1in}

\section{Restricting Odd Antiholes}

We begin by proving the following technical lemma.

\begin{lemma}\label{lem:xy} Let $G$ be a $k$-critical graph and suppose $G$ contains an odd antihole $A=\overline{C_{2t+1}}$. Let $\emptyset\ne X\subsetneq V(A)$. Then either there exists 
$x\in X$ satisfying
$$d_G(x)\geq k+|X|-3-\alpha(C_{2t+1}[X]),$$
or there exists $y\in V(A)\setminus X$ satisfying
$$d_G(y)\geq k+|X|-3.$$
\end{lemma}

\begin{proof} Suppose, for a contradiction, that no such $x$ or $y$ exists. Since $G$ is $k$-critical, $G-A$ has a $(k-1)$-coloring $\varphi$. For every vertex $v$ of the antihole, define
$$L(v)=\{1,\ldots,k-1\}\setminus\varphi(N_G(v)-V(A)).$$
Every vertex $v$ of $A$ has $2t-2$ neighbors in $A$, and hence
$$|L(v)|\ge(k-1)-\bigl(d_G(v)-(2t-2)\bigr)=k+2t-3-d_G(v).$$
Since $x,y$ do not exist, this means that when $v\in X$,
$$|L(v)|\ge 2t+1-|X|+\alpha(C_{2t+1}[X]),$$
and when $v\in V(A)\setminus X$,
$$|L(v)| \geq 2t+1-|X|.$$
Let $S=V(A)\setminus X$. Since $S\ne\emptyset$, the graph $C_{2t+1}-S=C_{2t+1}[X]$ is a forest and hence bipartite. We may now apply Lemma \ref{lem:listmatch} with this set $S$, where $|S|=2t+1-|X|$. This tells us that $A$ is $L$-colorable, which extends $\varphi$ to a $(k-1)$-coloring of $G$, a contradiction. 
\end{proof}

We will prove Theorem \ref{thm:degrees} using the above lemma and the following classic result of Gallai. Here, a \emph{block} is a maximal connected subgraph with no cut vertex.

\begin{theorem}[Gallai~\cite{Gallai}]\label{thm:gallai}
Let $G$ be a $k$-critical graph and let
$$D=\{v\in V(G):d_G(v)=k-1\}.$$
Then every block of $G[D]$ is complete or an odd cycle.
\end{theorem}

\setcounter{theorem}{5}
\begin{theorem} Let $G$ be a $k$-critical graph with maximum degree $\Delta$, and suppose $G$ contains an odd antihole $\overline{C_{2t+1}}$ with $t\geq 3$ (i.e., not $\overline{C_5}$). Then the number of vertices in the antihole with degree at least $k$ in $G$ is at least $t+1$ if $t\geq 4$ and at least $t$ if $t=3$. Moreover, for $1\leq j\leq t-3$,  the number of vertices in the antihole with degree at least $k+j$ in $G$ is at least  $\min\{t-j-1, 2t+k-\Delta-2\}$.
\end{theorem}
\setcounter{theorem}{14}

\begin{proof}  Let $A=\overline{C_{2t+1}}$ and consider some fixed $j\in\{1, \ldots, t-3\}$. Let 
$$D_j=\{v\in V(A):d_G(v)\ge k+j\}.$$
Apply Lemma~\ref{lem:xy} with $X_u=V(A)\setminus\{u\}$ for any $u\in V(A)$. Since $\alpha(C_{2t+1}[X_u])=t$, the first alternative of the lemma would give a vertex $x\in X_u$ with
$d_G(x)\ge k+|X_u|-3-t=k+t-3\geq k+j$; the second alternative would give 
$d_G(u)\geq k+|X_u|-3\ge k+ 2t-3$.  In particular, this means that $D_j\neq\emptyset$.

Suppose that $|D_j|< \min\{t-j-1, 2t+k-\Delta-2\}$. Let $X=V(A)\setminus D_j$, so $|X|=2t+1-|D_j|$ and $|X|\geq t+j+3$, $|X|\geq \Delta-k+4$.
If we apply Lemma~\ref{lem:xy} to this $X$, the second alternative of the lemma would give a vertex $y\in D_j$ satisfying 
$$d_G(y)\geq k+|X|-3\geq k+(\Delta-k+4)-3=\Delta+1$$
which is impossible. So in fact there must exist a vertex $x\in X$ with
$$d_G(x)\ge k+|X|-3-\alpha(C_{2t+1}[X])\geq k+|X|-3-t\geq  k+(t+j+3)-3-t\geq k+j.$$
Since $x\not\in D_j$, this is a contradiction.

It remains now to prove the first assertion of the theorem. We know that the minimum degree of $G$ is at least $k-1$; we will show that there are not too many vertices of $A$ achieving this lower bound. To this end, let
$D=\{v\in V(A):d_G(v)=k-1\}$ and let $D_G=\{v\in V(G):d_G(v)=k-1\}$. By Theorem~\ref{thm:gallai}, every block of $G[D_G]$ is complete or an odd cycle. If $B$ is a block of $G[D]$, then it is a subgraph of some block $B_G$ of $G[D_G]$. If $B_G$ is a clique, then so is $B$. If $B_G$ is an odd cycle, then the only way that $B$ is still a block is if $B$ is just $K_1$, $K_2$ or $B=B_G$. Hence every block of $G[D]=A[D]$ is a clique or odd cycle.

We aim to show that $|D|\leq t$ if $t\geq 4$ and $|D|\leq t+1$ if $t=3$, so suppose not, i.e., suppose $|D|\geq t+1$ if $t\geq 4$ and $|D|\geq t+2$ if $t=3$. In particular, this means that $|D|\geq 5$. It also means that $D$ must contain a pair of vertices that are consecutive on $C_{2t+1}$, and hence non-adjacent in $A[D]$. 

If $A[D]$ consists of only one block, then since it has a pair of non-adjacent vertices, $A[D]$ must be an odd cycle. But $\overline{C_{2t+1}}$ has no induced odd cycle of length at least five: if $A[Y]=C_r$ with $r\geq 5$, then $C_{2t+1}[Y]$ is $(r-3)$-regular, while every induced subgraph of $C_{2t+1}$ has maximum degree at most 2; hence $r=5$ which would force $C_{2t+1}[Y]=C_5$, impossible for a proper induced subgraph of a cycle. Thus $A[D]$ is a triangle, contradicting $|D|\geq 5$.

If $A[D]$ has multiple components, then $\overline{A[D]}$ must contain a complete bipartite graph (between one component of $A[D]$ and the rest of them). Since $\overline{A[D]}\subseteq C_{2t+1}$ it has maximum degree at most 2, and hence the complete bipartite graph cannot be larger than $K_{2,2}$. But this means that $|D|\leq 2+2$, a contradiction.

We now know that $A[D]$ is connected but has a cut vertex, say $z$, separating some partition $U, W$ of $A[D]\setminus z$. Then  $\overline{A[D]\setminus z}$ must contain a complete bipartite graph between $U$ and $W$, and applying the argument in the above paragraph tells us that $|U|=|W|=2$. But then $\overline{A[D]\setminus z}$ contains $K_{2, 2}=C_4$, which is not possible since $\overline{A[D]}\subseteq C_{2t+1}$ and hence cannot contain a $C_4$.
\end{proof}

We can now prove Theorems \ref{thm:length} and \ref{thm:2k-5}, which we restate here, the latter of which in an expanded form.

\setcounter{theorem}{1}
\begin{theorem} Let $G$ be a $k$-critical graph with maximum degree $\Delta$. Then the length of any odd antihole in $G$ is at most $2(\Delta-k)+7$. 
\end{theorem}
\begin{proof} Let $A=\overline{C_{2t+1}}$ be an odd antihole in $G$. Apply Lemma \ref{lem:xy} with the set $X_u=V(A)\setminus \{u\}$ for any $u\in V(A)$, as we did at the start of the proof of Theorem \ref{thm:degrees}. Like there, we get either a vertex $x\in X_u$ with $d_G(x)\geq k+t-3$, or we get that $d_G(u)\geq k+2t-3$. Either way, $\Delta\geq k+t-3$, so $t\leq \Delta-k+3$ and $2t+1\leq 2(\Delta-k)+7$.
\end{proof}

\setcounter{theorem}{3}
\begin{theorem} Let $G$ be an odd-hole-free $k$-critical graph with maximum degree $\Delta$. If $\Delta\leq k$ or $\Delta \geq 2k-5$, and $G$ contains the odd antihole $\overline{C_{2(\Delta-k)+7}}$, then $G=\overline{C_{2(\Delta-k)+7}}$.
\end{theorem}
\setcounter{theorem}{14}

\begin{proof} We already discussed the $\Delta\leq k$ case in the introduction, so we may assume that $\Delta\geq 2k-5$. Let $A=\overline{C_{2(\Delta-k)+7}}$ and set $t=\Delta-k+3$. Suppose for a contradiction that $G\neq A$. In particular, since $\chi(A)=t+1$ and $G$ is critical, this means that $\chi(G)=k\geq t+2$. Substituting for $t$, this tells us that
$$k \geq (\Delta-k+3)+2 \hspace*{.1in}\Rightarrow\hspace*{.1in} k\geq \tfrac{\Delta+5}{2}.$$
Since $\Delta\geq 2k-5$, we also have $k\leq \tfrac{\Delta+5}{2}$. Hence $k=\tfrac{\Delta+5}{2}=t+2$. 

Since $G$ is $k$-critical, let $\varphi$ be a $(t+1)$-coloring of $G-A$. For $v\in V(A)$, let
$$L(v)=\{1, 2, \ldots, t+1\}\setminus\varphi(N_G(v)-V(A)).$$
Note that $|L(v)|\geq t$ for all $v\in V(A)$, because $v$ has exactly $2t-2$ neighbors in $A$ and $\Delta=2t-1$. Since every vertex of $A$ has at most one neighbor outside $A$, each $L(v)$ has size $t$ or $t+1$. For every $v\in V(A)$, choose a $t$-subset $L'(v)\subseteq L(v)$, choosing them so that they are not all identical if possible. We can now apply Theorem \ref{thm:list} since $|V(A)|=2(t+1)-1$ and its complement (a cycle) is hypomatchable. If the lists are not all identical, then we get that $A$ is $L'$-colorable, and hence $L$-colorable, a contradiction. Thus it is impossible to choose $t$-subsets $L'(v)\subseteq L(v)$ that are not all identical. This forces every $L(v)$ itself to be the same $t$-set. This means that there is a unique color $c \in \{1, 2, \ldots, t+1\}$ such that: every $v\in V(A)$ is adjacent to exactly one vertex $z_v \in G\setminus A$, and $\varphi(z_v)=c$. Let $Z=\{z_v:v\in V(A)\}$.

Choose some $a\in \{1, 2, \ldots, t+1\}\setminus\{c\}$, and consider the components of $G-A$ induced by colors $a$ and $c$ under $\varphi$. If vertices of $Z$ occur in more than one such component, then let $\varphi'$ be the $(t+1)$-coloring of $G-A$ obtained from $\varphi$ by switching $a$ and $c$ on one of them. 
After switching $a$ and $c$ on the component, a vertex $v\in A$ whose unique neighbor $z_v$ lies in the switched component now has residual list $\{1, \ldots, t+1\}\setminus\{a\}$, while every other vertex has residual list $\{1, \ldots, t+1\}\setminus\{c\}$. Since the component contains some but not all vertices of $Z$, these $t$-lists are not all identical.
Then however we can again apply Theorem \ref{thm:list} to appropriate lists $L'$, which will now not be identical, and hence we can get a $(t+1)$-coloring of $G$, contradiction. So all vertices of $Z$ lie in the same $(a,c)$-component of $G-A$, say $B$.

We claim we can choose $x,y\in V(A)$ that are adjacent in $A$ (i.e., nonconsecutive on $C_{2t+1}$) and have $z_x\neq z_y$. If not, then indeed $|Z|=1$ (since $A$ is connected) which is impossible since that single vertex would have to have degree at least $2t+1=\Delta+2$. Given such $x,y$, let $P$ be a shortest $(z_x, z_y)$-path in $B$. Since $z_x\neq z_y$ are both colored $c$ under $\varphi$, $P$ has even length at least 2. If two vertices of $P$ are adjacent in $G$ then they receive different colors under $\varphi$, so one is colored $a$ and the other $c$, and the edge joining them lies in $B$; by the minimality of $P$ it is an edge of $P$. Hence $P$ is an induced path in $G$. Since $z_x$ is the unique neighbor of $x$ outside $A$, the only edges between $\{x,y\}$ and $V(P)$ are $xz_x$ and $yz_y$. So $P$ together with the path $(z_x, x, y, z_y)$ is an induced odd cycle of length at least five, that is, an odd hole in $G$, a contradiction.
\end{proof}

If we are free to use Conjecture \ref{conj:skcc} in place of Theorem \ref{thm:list} in the above proof, then we can get all the way to Conjecture \ref{conj:main}, as we now prove.

\begin{proof}\emph{(Proof of Theorem \ref{thm:conjs})} We follow the argument in the proof of Theorem \ref{thm:2k-5}. Let $A=\overline{C_{2\ell-1}}$ with $\ell=\Delta-k+4$, and suppose that $G\neq A$ and that $G$ contains no odd holes. As in the proof of Theorem \ref{thm:2k-5}, $k\geq \ell+1$. Set $q=k-\ell\geq 1$. Then 
\begin{equation}\label{eqn:tag1}
k=\ell+q \textrm{ and } \Delta=2\ell+q-4.
\end{equation}

Let $\phi$ be a $(k-1)$-coloring of $G-A$, with color set $P$, so $|P|=\ell-1+q$, and define the residual lists $L(v)=P\setminus \phi(N_G(v)\setminus V(A))$ for all $v\in V(A)$. Since each $v\in V(A)$ has $2\ell-4$ neighbors in $A$, (\ref{eqn:tag1}) implies that $|L(v)|\geq \ell-1$.

Choose an $(\ell-1)$-subset $L'(v)\subseteq L(v)$ for each $v\in V(A)$. If these can be chosen so that they are not all identical, then since we are assuming that $\overline{C_{2\ell-1}}$ is strongly $\ell$-chromatic-choosable, we get an $L'$-coloring of $A$, contradicting the $k$-criticality of $G$. Thus every possible choice of the $L'(v)$ is identical, which forces $L(v)=S$ for all $v\in V(A)$ for some fixed $(\ell-1)$-set $S\subseteq P$.

Let $F=P\setminus S$, so $|F|=q$. By $L(v)=S$, every vertex $v\in V(A)$ has an outside neighbor of every color in $F$. Since $v$ has at most $q$ neighbors outside $A$, it has exactly one such neighbor of each color in $F$, and no outside neighbor whose color lies in $S$. For $c\in F$, denote the unique $c$-colored neighbor of $v$ by $z_v(c)$.

Fix $c\in F$ and $a\in S$. The same Kempe-switch argument used in the proof of Theorem \ref{thm:2k-5} shows that all vertices $z_v(c)$ for all $v\in V(A)$, lie in the same $(a,c)$-component of $G\setminus A$: otherwise switching $a$ and $c$ on a component containing some but not all of them produces the two residual $(\ell-1)$-lists $S$ and $S\setminus\{a\}\cup\{c\}$, and Conjecture \ref{conj:skcc} again yields a $(k-1)$-coloring of $G$.

The same shortest-path argument from the proof of Theorem \ref{thm:2k-5} now shows that, for fixed $c\in F$, all the vertices $z_v(c)$ are in fact equal. Otherwise, since $A$ is connected, there are adjacent $x,y\in V(A)$ with $z_x(c)\neq z_y(c)$. 
Let $T$ be a shortest $(z_x(c),z_y(c))$-path in their common $(a,c)$-component. Exactly as in the proof of Theorem \ref{thm:2k-5}, $T$ is an induced path in $G$ of even length at least 2. Every vertex of $T$ is colored $a$ or $c$. Since $a\in S$, the vertex $x$ has no outside neighbor colored $a$, and $z_x(c)$ is its unique outside neighbor colored $c$; hence $z_x(c)$ is the only vertex of $T$ adjacent to $x$. Likewise $z_y(c)$ is the only vertex of $T$ adjacent to $y$. Hence $T$ together with the path $(z_x(c),x,y,z_y(c))$ is an odd hole in $G$, a contradiction. Thus for each $c\in F$ there is a vertex $z_c$ adjacent to every vertex of $A$.

Set $Z=\{z_c: c\in F\}$. The vertices of $Z$ are distinct, so $|Z|=q$ and $N_G(v)\setminus V(A)=Z$ for every $v\in V(A)$.

Every $z\in Z$ already has $(2\ell-1)$ neighbors in $A$, and hence, by (\ref{eqn:tag1}) we get that 
\begin{equation}\label{eqn:tag4}
d_{G\setminus A}(z)\leq \Delta-(2\ell-1)=q-3.
\end{equation} 
If $q\leq 2$, this is impossible. Hence $q\geq 3$.

Let $R=G-(A\cup Z)$, and retain the coloring $\phi$ on $R$. Choose any $(q-1)$-set $Q\subseteq P$. For all $z\in Z$, let $M(z)=Q\setminus \phi(N_R(z))$. By (\ref{eqn:tag4}), 
$$|M(z)|\geq q-1-d_R(z)\geq d_{G[Z]}(z)+2.$$
Thus $G[Z]$ can be greedily colored from the lists $M$. We have therefore recolored $G-A$ so that $Z$ uses at most $q-1$ colors. 

Finally, since $N_G(v)\setminus V(A)=Z$ for every $v\in V(A)$, every vertex of $A$ has precisely $Z$ as its outside neighborhood. Hence all vertices of $A$ have a common set of at least $(k-1)-(q-1)=\ell$ available colors. Since $\chi(A)=\ell$, the coloring extends to $A$, contradicting the criticality of $G$.
\end{proof}

\vspace*{.1in}

\section{Odd-hole-free attachments and sharpness}

We start with the following technical lemma.

\begin{lemma}\label{lem:nbhd} Let $A=\overline{C_{2t+1}}$, $t\geq 3$ be an induced subgraph in $G$, and let $C=C_{2t+1}$ be the cycle whose complement is $A$. For $x\in V(G)\setminus V(A)$, put $X=N_G(x)\cap V(A)$. Then $G[V(A)\cup\{x\}]$ is odd-hole-free if and only if no component of $C[X]$ has order two.
\end{lemma}

\begin{proof}
First, $A$ has no odd hole. Indeed, if $A[Y]$ were an induced odd cycle of order $r\ge5$, then every vertex of $C[Y]$ would have degree $r-3$. Since $\Delta(C[Y])\le2$, we would have $r=5$ and $C[Y]\cong C_5$, which is impossible because every proper induced subgraph of $C$ is a forest and $|V(C)|\ge7$. Label the vertices of $C$ cyclically as $v_0,\ldots,v_{2t}$. Suppose first that $C[X]$ has a component of order two, say $v_{i+1}v_{i+2}$. Then $x$ is adjacent to $v_{i+1}$ and $v_{i+2}$ and non-adjacent to $v_i$ and $v_{i+3}$. In $A$, the vertices $v_{i+1},v_{i+3},v_i,v_{i+2}$ induce a path in this order. Hence
$$
xv_{i+1}v_{i+3}v_iv_{i+2}x
$$
is an induced odd hole.

Conversely, suppose that $G[V(A)\cup\{x\}]$ contains an odd hole $D$. Since $A$ has no odd hole, we have $x\in V(D)$. Then $D-x$ is an induced path in $A$. If this path has $p$ vertices, then its complement is an induced subgraph of $C$ isomorphic to $\overline{P_p}$. Since every induced subgraph of $C$ has maximum degree at most two and an endpoint of $P_p$ has degree $p-2$ in $\overline{P_p}$, we have $p\le4$. Since $D$ is an odd hole, $p+1$ is odd and at least five, so $p=4$. Thus the four vertices of $D-x$ are consecutive on $C$. The two internal vertices of this four-vertex path in $C$ are adjacent to $x$, while its two ends are not. Therefore these two internal vertices form a component of order two in $C[X]$. This proves the claim.
\end{proof}

We can now prove our sharpness result, which we restate here for convenience.

\setcounter{theorem}{4}
\begin{theorem} For every $k\geq 5$, there exists a $k$-critical odd-hole-free graph $G$ with maximum degree $\Delta=2k-4$, and containing an induced odd antihole $\overline{C_{2k-3}}=\overline{C_{2(\Delta-k)+5}}$, but no shorter induced odd antihole.
\end{theorem}

\begin{proof} Let $k\ge5$, and put $s=k-2$. Label the vertices of $C_{2s+1}$ cyclically as $v_0,\ldots,v_{2s}$. Let $A=\overline{C_{2s+1}}$, and obtain $G$ from $A$ by adding two non-adjacent vertices $x$ and $y$ such that
$$
N_G(x)\cap V(A)=V(A)\setminus\{v_1,v_3\}
$$
and
$$
N_G(y)\cap V(A)=V(A)\setminus\{v_0,v_2\}.
$$

We first show that $\chi(G)=k$. Since $\chi(A)=s+1$, coloring $A$ with $s+1$ colors and giving $x$ and $y$ one new common color shows that $\chi(G)\le s+2$. In every $(s+1)$-coloring of $A$, one color class is a singleton and the other $s$ color classes are edges of $C_{2s+1}$. Since $v_1$ and $v_3$ are adjacent in $A$, the vertex $x$ can receive a color used on $A$ only when the singleton is $v_1$ or $v_3$. Similarly, $y$ can receive a color used on $A$ only when the singleton is $v_0$ or $v_2$. These two sets are disjoint, so $G$ has no $(s+1)$-coloring. Hence
$$
\chi(G)=s+2=k.
$$

We now show that the graph $G$ is $k$-critical. Indeed, the perfect matching of $C_{2s+1}-v_0$ gives an $(s+1)$-coloring of $A$ in which $v_0$ is the unique singleton, so this coloring extends to $G-x$ by giving $y$ the color of $v_0$. Similarly, an $(s+1)$-coloring of $A$ with singleton $v_1$ extends to $G-y$. If $v\in V(A)$, then the perfect matching of $C_{2s+1}-v$ gives an $s$-coloring of $A-v$, and giving $x$ and $y$ one new common color gives an $(s+1)$-coloring of $G-v$.

We now compute the maximum degree of $G$. To this end, note that every vertex of $A$ has degree $2s-2$ in $A$. The vertices $v_0,v_1,v_2,v_3$ gain one neighbor, every other vertex of $A$ gains two neighbors, and
$d_G(x)=d_G(y)=2s-1.$
Therefore
$\Delta(G)=2s=2k-4.$

We now show that the graph $G$ has no odd hole. First, 
$$C[N_G(x)\cap V(A)]= C[V(A)\setminus \{v_1, v_3\}]$$
has no component of order two; likewise
$$C[N_G(y)\cap V(A)]= C[V(A)\setminus \{v_0, v_2\}]$$
has no component of order two. Thus, by Lemma \ref{lem:nbhd}, no odd hole contains exactly one of $x$ and $y$. Suppose that an odd hole contains both. Since $x$ and $y$ are non-adjacent, each has exactly two neighbors on the hole. Since each has only two non-neighbors in $A$, the hole has at most six vertices, and hence, being odd, it must be a $5$-cycle. Write it as
$$
xaybcx.
$$
Then $b\in\{v_1,v_3\}$, $c\in\{v_0,v_2\}$, and $a$ is adjacent in $C_{2s+1}$ to both $b$ and $c$. This is impossible, since no vertex of $C_{2s+1}$ has one neighbor in each of the sets $\{v_1,v_3\}$ and $\{v_0,v_2\}$.

Finally, we show that $G$ contains no smaller odd antihole. In $\overline{G}$, the graph induced by $V(A)$ is $C_{2s+1}$, and the remaining edges are
$$
xy,\quad xv_1,\quad xv_3,\quad yv_0,\quad yv_2.
$$
Let $D$ be an induced odd cycle in $\overline{G}$ of order at least five. If $D$ contains neither $x$ nor $y$, then $D=C_{2s+1}$. If it contains exactly $x$, then its two neighbors on $D$ are $v_1$ and $v_3$; the two $v_1$--$v_3$ paths in $C_{2s+1}$ give a $4$-cycle or a cycle of order $2s+1$. The same holds if $D$ contains exactly $y$. If $D$ contains both $x$ and $y$, then $xy\in E(D)$. The remaining path has one end in $\{v_1,v_3\}$ and the other in $\{v_0,v_2\}$. The three adjacent endpoint pairs give $4$-cycles, while the pair $v_3,v_0$ gives the cycle
$$
xyv_0v_{2s}v_{2s-1}\cdots v_3x
$$
of order $2s+1$; every other path has a chord from $x$ or $y$. Thus every induced odd cycle of $\overline{G}$ of order at least five has order $2s+1$. Equivalently, $G$ contains no induced odd antihole $\overline{C_{2q+1}}$ with $3\leq q<s$. Moreover, $G$ contains no $\overline{C}_5=C_5$, since we have already shown that $G$ is odd-hole-free. Hence $A=\overline{C_{2s+1}}$ is a shortest induced odd antihole in $G$.
Since $\Delta=2k-4$,
$$
\Delta-k+2=k-2=s.
$$
So  $2s+1=2k-3=2(\Delta-k)+5$ and both terms of the bound are attained.
\end{proof}

\end{document}